\documentclass[12pt]{amsart}
\usepackage{graphicx,amssymb,stmaryrd,amsfonts,epsfig,amsthm,a4,amsmath,mathrsfs,url,
enumerate,amsgen,amsmath,amstext,amsbsy,amsopn}
\usepackage{vmargin,amscd}
\usepackage[usenames,dvipsnames]{color}

\newtheorem{thm}{Theorem}[section]
\newtheorem{cor}[thm]{Corollary}
\newtheorem{lem}[thm]{Lemma}
\newtheorem{prop}[thm]{Proposition}

\theoremstyle{definition}
\newtheorem{defn}[thm]{Definition}

\newtheorem{exe}[thm]{Example}

\theoremstyle{remark}
\numberwithin{equation}{section}

\newcommand{\eps}{\varepsilon}
\newcommand{\Rips}{\textnormal{Rips}}
\newcommand{\SL}{\textnormal{SL}}
\newcommand{\GL}{\textnormal{GL}}
\newcommand{\Z}{\mathbf{Z}}

\newcommand{\R}{\mathbf{R}}

\newcommand{\Q}{\mathbf{Q}}
\newcommand{\K}{\mathbf{K}}

\newcommand{\HNN}{\mathrm{HNN}}

\title{Zooming in on the large-scale geometry of locally compact groups}

\author{Yves de Cornulier and Pierre de la Harpe}

\date{December 2, 2015 -- revised April 15, 2016}



\begin{document}

\maketitle

\section{Introduction}

The purpose of this survey is to describe how locally compact groups 
can be studied as geometric objects. 
We will emphasize the main ideas and skip or just sketch most proofs, 
often referring the reader to our much more detailed book \cite{CH--16}.

It is now classical to view any finitely generated group $G$ as a geometric object. 
If $S$ is a finite generating subset, 
the Cayley graph $\Gamma(G,S)$ is the graph whose vertex set is $G$, 
and whose edges are the pairs $(g,gs)$ where $(g,s)$ ranges over $G \times S$. 
Note that this can be viewed as an oriented and labeled graph, 
possibly with self-loops (if $1 \in S$), but this does not play any role in the sequel. 
Since $S$ generates $G$, this graph is connected, 
and therefore there is a well-defined metric on the vertex set $G$, 
for which the distance $d_S(g,h)$ between $g,h \in G$ 
is the number $\ell_S(g^{-1}h)$ defined as the smallest $k$ 
such that $g^{-1}h$ can be written as a product of $k$ elements in $S^{\pm 1}$. 
This metric, called the {\bf word metric} (with respect to $S$), enjoys the following properties:

\begin{enumerate}
\item 
it is left-invariant, i.e.\ the left action of $G$ on itself is by isometries;
\item 
it is proper, in the sense that bounded subsets are finite;
\item 
it satisfies the following geodesicity property: 
for all integers $n$ and $g,h \in G$ such that $d_S(g,h)=n$, 
there exist $g_0,g_1,\dots,g_n \in G$ such that $d_S(g_{i-1},g_{i})=1$ 
for all $i=1,\dots,n$ and $(g_0,g_n)=(g,h)$.
\end{enumerate}

The main drawback of this metric is that it depends on the choice of a finite generating subset $S$; 
in particular, a metric property of this metric need not be intrinsic to $G$. 
Nevertheless, if $S'$ is another finite generating subset, an easy induction shows that, 
for some constants $c,c'>0$, we have
\[cd_{S'} \le d_S \le c'd_{S'}.\]
In other words, the identity $(G,d_S)\to (G,d_{S'})$ is a {\bf bilipschitz map}.  
 
Word metrics on finitely generated groups have proved useful on several occasions,
for example in \cite{Dehn--11}, \cite{Svar--55, Miln--68}, and \cite{Grom--81, Grom--84, Grom--93}.
It is natural to wonder how this concept generalizes to a broader setting. 
 
Let us first discuss a generalization to discrete groups, beyond finitely generated ones. 
If we consider the word metric with respect to a generating subset $S$, then obviously $S$ is bounded; 
if we require the properness of the metric, $S$ is necessarily finite, 
so, in a sense, the word metric is only suitable for finitely generated groups. 
Beyond the case of word metrics, it is natural to consider left-invariant proper metrics. 
For instance, for a group $G$ with a finite generating subset $S$ 
and a subgroup $H \subset G$, 
the restriction of $d_S$ to $H$ is a left-invariant proper metric on $H$. 
This is not, in general, a word metric on $H$. 
Actually, by a theorem of Higman-Neumann-Neumann \cite{HiNN--49}, 
every countable group is isomorphic to a subgroup of a finitely generated group, 
and thus admits a left-invariant proper metric by the above construction. 
Conversely, it is clear that the existence of a proper metric implies the countability of the group. 
The uniqueness up to bilipschitz maps fails for infinite groups, 
since when $d$ is a proper left-invariant metric on $G$ then so is $\sqrt{d}$, and $(G,\sqrt{d})$ can be checked to never be bilipschitz (nor quasi-isometric, see Definition \ref{defqi}) to $(G,d)$. 
However, if $G$ is a discrete group with two proper left-invariant metrics $d,d'$, 
there exist nondecreasing functions $\Phi_-,\Phi_+$ from the set of nonnegative numbers to itself, 
tending to $+\infty$ at $+\infty$, such that $\Phi_-\circ d' \le d \le \Phi_+\circ d'$. 
This is interpreted by saying that the identity map $(G,d)\to (G,d')$ is a {\bf coarse equivalence}.
 
A further generalization is to consider topological groups, especially locally compact groups. 
Given a topological group $G$, we consider metrics (or pseudo-metrics) on $G$; 
we {\em do not} consider the topology defined by these metrics and only refer to the given topology on $G$.
It is natural to require that compact subsets are bounded 
(noting that this is automatic when the metric or pseudo-metric is continuous). 
The properness assumption is that bounded subsets have a compact closure. 
A convenient setting is to assume the topological group to be locally compact 
(and in particular, Hausdorff, by definition); 
note that this includes discrete groups as an important particular case. 
Furthermore, to avoid local topological issues, we allow pseudo-metrics.

In \S\ref{mec}, we introduce some general metric notions, 
including coarsely Lipschitz maps and coarse equivalences. 
In \S\ref{colco}, we define the coarse language in the context of locally compact groups,
and we characterize $\sigma$-compact locally compact groups in a metric way.
In \S\ref{geno}, we introduce the coarse and large-scale geodesic notions, 
which allow to characterize compactly generated locally compact groups in a metric way. 
In \S\ref{selli}, we introduce coarsely ultrametric spaces, 
which provide a coarse characterization of locally elliptic locally compact groups, 
which generalize locally finite groups from the discrete setting. 
In \S\ref{copr}, we introduce the notion of coarse properness for metric spaces, 
which allows to define in a coarse setting the notions of growth and amenability. 
In \S\ref{coprg}, we introduce coarsely simply connected metric spaces 
and use them to characterize metrically compactly presented groups, 
which generalize finitely presented groups in the setting of locally compact groups. 
In the last two sections, we illustrate compact presentability: 
in \S\ref{bst}, we describe the Bieri-Strebel Theorem, 
which provides constraints for surjective homomorphisms of compactly presented locally groups onto $\Z$, 
and \S\ref{sexam} provides further examples.

\section{Metric categories}
\label{mec}
We denote by $\R_+$ the set of nonnegative real numbers. 
The standard metric $d$ is defined on $\R_+$ by $d(x,y) = \vert y-x \vert$.
 
The objects we will consider are pseudo-metric spaces, 
that is, pairs $(X,d)$ where $d$ is a symmetric function $X \times X \to \R_+$, 
satisfying the triangle inequality. 
By a common abuse of notation, a pair $(X,d)$ will often be identified with the underlying set $X$.

\begin{defn}
A map $f:X\to Y$ between pseudo-metric spaces is
\begin{itemize}
\item 
{\bf coarsely Lipschitz} if there exists a nondecreasing map $\Phi_+ : \R_+ \to \R_+$ 
such that $d(f(x),f(x')) \le \Phi_+(d(x,x'))$ for all $x,x' \in X$; 
we say that $f$ is $\Phi_+$-coarse;
\item 
{\bf large-scale Lipschitz} if it is $\Phi_+$-coarse for some affine function $\Phi_+$.
\end{itemize}
Two maps $f,f':X\to Y$ are {\bf close}, written $f\sim f'$, if $\sup_{x \in X}d(f(x),f'(x)) < \infty$. Equivalence classes of this equivalence relation are called {\bf closeness classes}.
\end{defn}

For instance, the map $f_a:\R_+\to\R_+$ mapping $x$ to $x^a$ ($a>0$) is coarse 
if and only if it is large-scale Lipschitz, if and only if $a \le 1$. If $(X,d)$ is an arbitrary unbounded metric space, 
then the identity map $(X,d)\to (X,\sqrt{d})$ is large-scale Lipschitz (hence coarse), 
while its inverse is coarse but not large-scale Lipschitz.

It is clear that, if $f$ is coarse (respectively large-scale Lipschitz) and $f\sim f'$, then $f'$ satisfies the same property.

\begin{defn}
The {\bf metric coarse category} (resp.\ {\bf large-scale category}) 
is the category whose objects are pseudo-metric spaces 
and morphisms are closeness classes of coarsely Lipschitz maps 
(resp.\ of large-scale Lipschitz maps). 
\end{defn}

\begin{defn}\label{defqi}
Let $f:X\to Y$ be a map between pseudo-metric spaces.
\begin{itemize}
\item 
The map $f$ is {\bf essentially surjective} if $\sup_{y \in Y} d(y,f(X)) < \infty$.
\item The map $f$ is {\bf coarsely expansive} 
if there exists a non-decreasing function\footnote{One 
could equally consider functions $\Phi_-  :  \R_+ \to \R_+ \cup \{\infty\}$.
This would not change the definition.}
$\Phi_- : \R_+ \to \R_+$ tending to infinity at infinity, 
such that $d(f(x),f(x')) \ge \Phi_-(d(x,x'))$ for all $x,x' \in X$; 
we say that $f$ is $\Phi_-$-coarsely expansive.
\item 
The map $f$ is {\bf large-scale expansive} 
if it is $\Phi$-coarsely expansive for some affine function $\Phi$. 
\item 
The map $f$ is a {\bf coarse equivalence} if it is coarse, coarsely expansive and essentially surjective.
\item 
The map $f$ is a {\bf quasi-isometry}
if it is large-scale Lipschitz, large-scale expansive and essentially surjective.
\item 
Two metric spaces $X,Y$ are {\bf coarsely equivalent} (resp.\ {\bf quasi-isometric}) 
if there exists a coarse equivalence (resp.\ quasi-isometry) $X\to Y$.
\end{itemize}
\end{defn}

\begin{prop}
\label{CategoricalCoarseQI}
Let $f : X \to Y$ be a map between pseudo-metric spaces.
\begin{enumerate}
\item 
$f$ induces an isomorphism in the metric coarse category if and only if $f$ is a coarse equivalence;
\item 
$f$ induces an isomorphism in the large-scale category if and only if $f$ is a quasi-isometry.
\end{enumerate} 
In particular, to be coarsely equivalent (resp.\ quasi-isometric) 
is an equivalence relation between pseudo-metric spaces.
\end{prop}

\begin{exe}
\begin{enumerate}
\item 
Let $X$ be a pseudo-metric space and let $\widehat{X}$ be its Hausdorffization, 
namely the metric space obtained from $X$ by identifying points at distance 0. 
Then the canonical projection $X \to \widehat{X}$ is a quasi-isometry (and hence a coarse equivalence). 
Thus, in the metric coarse category and in the large-scale category, 
the full subcategories where objects are metric spaces are essential.
\item 
Let $X$ be a metric space and $\overline{X}$ its completion. 
Then the canonical injection $X \to \overline{X}$ is a quasi-isometry 
(it is indeed an isometry onto a dense subset).
\item 
Let $X$ be a pseudo-metric space and let $Y \subset X$ be a subset 
maximal for the property that any two points in $Y$ have distance $\ge 1$. 
Then the isometric injection $Y \subset X$ is essentially surjective and thus is a quasi-isometry. 
Thus every metric space is quasi-isometric to a discrete one.
\item 
To be bounded is invariant under coarse equivalence, 
and all non-empty bounded pseudo-metric spaces are quasi-isometric.
\end{enumerate}
\end{exe} 

\noindent
\textbf{On proofs.}
The verification of the claims of Proposition \ref{CategoricalCoarseQI} is a routine exercise.
See in \cite{CH--16} Section 3.A, in particular Propositions 3.A.16 and 3.A.22.
 
\section{Coarse category of locally compact groups}
\label{colco}

\subsection{The abstract coarse category}
\label{clc}

\begin{defn}
\label{defcmcc}
Let $G, H$ be locally compact groups and $f : G \to H$ a map
(not necessarily a homomorphism or continuous).

Then $f$ is a {\bf coarse map} if, for every compact subset $K \subset G$, 
there exists a compact subset $L \subset H$ such that, for all $g,g' \in G$, 
the relation $g^{-1}g' \in K$ implies $f(g)^{-1}f(g') \in L$.

Let $f' : G \to H$ be another map. 
Then $f$ and $f'$ are {\bf close} 
if the set $\{ h \in H \mid h = f(g)^{-1}f'(g)\hskip.2cm \text{for some} \hskip.2cm g \in G\}$ 
has a compact closure. 
Equivalence classes of this equivalence relation are called {\bf closeness classes}.
\end{defn}
 
For instance, 
any continuous homomorphism between locally compact groups is a coarse map. 
If we have a semidirect product of locally compact groups $G=H\rtimes K$ with $K$ compact, 
then the mapping $hk\mapsto h$, for $(h,k) \in H \times K$,
is close to the identity of $G$ (but is in general not a homomorphism).

\begin{defn}
The {\bf coarse category of locally compact groups} 
is the category in which objects are locally compact groups 
and morpisms are closeness classes of coarse maps.
\end{defn}
 
\begin{defn}
\label{defesce}
Let $G, H$ be locally compact groups and $f : G \to H$ a map.
\par
Then $f$ is {\bf essentially surjective} 
if there exists a compact subset $L \subset H$ such that 
$H = f(G)L := \{ f(g)l \mid g \in G, l \in L \}$,
\par
and $f$ is {\bf coarsely expansive} if, 
for every compact subset $L \subset H$, there exists a compact subset $K \subset G$ 
such that, for all $g,g' \in G$, the relation $g^{-1}g' \notin K$ implies $f(g)^{-1}f(g')\notin L$.
\end{defn}

When $\sigma$-compact locally compact groups are treated as metric objects, 
the terminology of Definitions \ref{defcmcc} to \ref{defesce}
could be in conflict with the metric notions of \S\ref{mec}; 
nevertheless Proposition \ref{CompCoarseEtCoaremetriquePourG} 
will show that these are equivalent notions.

\begin{prop}
\label{GroupsMapsCoarseIso}
A map $f:G\to H$ between locally compact groups 
induces an isomorphism in the coarse category 
if and only if $f$ is a coarse map, is coarsely expansive, and is essentially surjective.
\end{prop}
 
An important example is the case of continuous homomorphisms.

\begin{prop}
\label{GroupsHomosCoarseIso}
Let $f : G \to H$ be a continuous homomorphism of locally compact groups. 
\par
Then $f$ is a coarse map. 
It is coarsely expansive if and only if it is proper, 
i.e.\  if and only if it has a compact kernel and a closed image.
It is essentially surjective if and only if $H/\overline{f(G)}$ is compact. 
\par
In particular, 
$f$ induces an isomorphism in the coarse category 
if and only it is proper and has a cocompact image.
\end{prop} 

Recall that a continuous map between locally compact topological spaces 
is proper if the inverse image of every compact subset is compact. 
Let $G, H$ be locally compact groups;
if a  continuous homomorphism $G \to H$ is proper, 
then it has a compact kernel and a closed image;
when $G$ is moreover $\sigma$-compact, 
the converse is true
(this follows from a result of Freudenthal, 
see Corollary 2.D.6 of \cite{CH--16}).

Being $\sigma$-compact is a coarse invariant among locally compact groups:

\begin{prop}
\label{SigmacomopactAndCE}
If $G$ and $H$ are coarsely equivalent locally compact groups 
and $G$ is $\sigma$-compact 
then so is $H$.
\end{prop}

\noindent
\textbf{On proofs.}
For one implication in Proposition \ref{GroupsMapsCoarseIso},
assume that $f : G \to H$ 
is coarse, coarsely expansive, and essentially surjective.
Let $L$ be a compact subset of $H$ such that $H = f(G)L$.
For every $y \in H$, choose $x_y \in G$ such that $y \in f(x_y)L$
and set $h(y) = x_y$.
It is elementary to check that $h : H \to G$
is well-defined up to closeness, coarse, 
and that its closeness class is the inverse of that of $f$. 
What remains to prove for Propositions
\ref{GroupsMapsCoarseIso}, \ref{GroupsHomosCoarseIso}
and \ref{SigmacomopactAndCE} 
is left as an exercise for the reader.

\subsection{Locally compact groups as pseudo-metric spaces} 

The abstract coarse theory can be expressed using the language of pseudo-metric spaces, under an extra assumption on the locally compact groups, namely when they are $\sigma$-compact, that is, are countable unions of compact subsets. This includes most familiar examples.

\begin{defn}
\label{apd}
Let $G$ be a locally compact group. 
An {\bf adapted pseudo-metric} on $G$ is a pseudo-metric which is
\begin{itemize}
\item left-invariant, 
\item locally bounded (compact subsets are bounded);
\item proper (bounded subsets have a compact closure).
\end{itemize}\end{defn}

\begin{thm}
\label{ExistAdapted}
A locally compact group admits an adapted pseudo-metric 
if and only if it is $\sigma$-compact.
\end{thm} 

\begin{proof}[Proof (sketch)]
One direction is clear. 
Conversely, suppose that $G$ is $\sigma$-compact 
and write $G=\bigcup K_n$ with $K_n$ a compact subset, 
contained in the interior of $K_{n+1}$. 
Consider the metric graph with $G$ as set of vertices 
and an edge $(g,gs)$ of length $n$ for all $n$ and every $(g,s) \in G \times K_n$. 
Then this graph is connected (since there is at least one edge between any two vertices), 
the graph metric on the set $G$ of vertices is left-invariant, 
and each compact subset of $G$ is bounded, being contained in some $K_n$. 
Moreover, bounded subsets have a compact closure: 
indeed, for every $n\ge 1$ the $n$-ball around $1$ 
is contained in the union of the $K_{n_1} \cdots K_{n_k}$, 
where $k\ge 1$ and $(n_1,\dots,n_k)$ ranges over the $k$-tuples of positive integers with sum $n$. 
Thus $G$ admits an adapted pseudo-metric 
(indeed a metric, since any two distinct points are at distance $\ge 1$). 
\end{proof}

\begin{prop}
\label{TwoAdaptedOnG}
Let $G$ be a $\sigma$-compact locally compact group. 
For any two adapted pseudo-metrics $d,d'$ on $G$, 
the identity map of pseudo-metric spaces $(G,d)\to (G,d')$ is a coarse equivalence.
\end{prop} 

If $G$ is a $\sigma$-compact locally compact group, 
it admits an adapted pseudo-metric $d$ by Theorem \ref{ExistAdapted},
and this allows to view $(G,d)$ as a well-defined object in the metric coarse category.

The following proposition shows that, for $\sigma$-compact locally compact groups, 
on which the definitions of \S\ref{mec} and \S\ref{clc} both make sense, the definitions are consistent.

\begin{prop}
\label{CompCoarseEtCoaremetriquePourG}
If $(G,d)$ and $(G',d')$ are $\sigma$-compact locally compact groups 
with adapted pseudo-metrics, 
a map $f:G\to G'$ is a coarse map of locally compact groups (in the sense of \S\ref{clc}) 
if and only if is a coarsely Lipschitz map of pseudo-metric spaces (in the sense of \S\ref{mec}). 
The same holds for coarsely expansive maps, essentially surjective maps, coarse equivalences, and closeness.
\end{prop}

\noindent
\textbf{More on proofs.}
We leave the proof of 
Proposition \ref{CompCoarseEtCoaremetriquePourG}
as an exercise for the reader.
For Theorem \ref{ExistAdapted} and Proposition \ref{TwoAdaptedOnG}, 
see Proposition 4.A.2 and Corollary 4.A.6 in \cite{CH--16}.
 
\section{Geodesic metric notions and compactly generated groups}
\label{geno}

\subsection{Coarse connectedness and geodesic notions}

We turn back to the metric setting,
and we provide, in the coarse setting, a characterization
of compactly generated locally compact groups 
among $\sigma$-compact locally compact groups.

\begin{defn}
\label{d_cg}
A pseudo-metric space $X$ is {\bf coarsely connected} 
if there exist $c>0$ such that the equivalence relation 
generated by ``being at distance at most $c$" identifies all points in $X$. 
That is, for any two points $x,y \in X$, 
there exist $n$ and $x=x_0,x_1,\dots,x_n=y$ in $X$ 
with $\sup_{1 \le i \le n}d(x_{i-1},x_{i}) \le c$.

The pseudo-metric space $X$ is {\bf coarsely geodesic} 
if there exists a nondecreasing function $\Phi:\R_+\to\R$ and $c>0$ such that, 
for any two points $x,y \in X$, 
there exist $n \le \Phi(d(x,y))$ and $x=x_0,x_1,\dots,x_n=y$ in $X$ 
with $\sup_{1 \le i \le n}d(x_{i-1},x_{i}) \le c$. 
It is {\bf large-scale geodesic} if the above $(\Phi,c)$ can be chosen with $\Phi$ an affine function.
\end{defn}
 
A basic observation is that being coarsely connected or coarsely geodesic 
are coarse invariants.
For instance, if there is a coarse equivalence 
between a coarsely geodesic pseudo-metric space and another pseudo-metric space, 
then the latter is coarsely geodesic as well.
 
Similarly, being large-scale geodesic is a quasi-isometry invariant. 
However, it is not a coarse invariant: 
if $(X,d)$ is an unbounded large-scale geodesic metric space, 
then it is coarsely equivalent to $(X,\sqrt{d})$, but the latter is not large-scale geodesic. 
It can actually be checked that a pseudo-metric space is coarsely geodesic 
if and only if it is coarsely equivalent to a large-scale geodesic metric space.

We saw in \S\ref{mec} examples of coarsely Lipschitz maps that are not large-scale Lipschitz. 
Nonetheless, we have the following useful proposition.

\begin{thm}
\label{LargescalegeoPlusCoarselyLip}
Let $f: X \to Y$ be a map between pseudo-metric spaces. 
Assume that $X$ is large-scale geodesic and that $f$ is a coarsely Lipschitz map. 

Then $f$ is large-scale Lipschitz.
\end{thm}

\begin{cor}
\label{LargescalegeoCoarse=QI}
Let $f : X \to Y$ be a coarse equivalence between  large-scale geodesic pseudo-metric spaces.  
Then $f$ is a quasi-isometry.
\end{cor}

\subsection{Compactly generated locally compact groups}

By definition, a locally compact group is compactly generated if it is generated, as a group, by a compact subset. For instance, for a discrete group it means being finitely generated.

\begin{thm}
\label{CaractCg}
Let $G$ be a $\sigma$-compact locally compact group 
and $d$ an adapted pseudo-metric on $G$. 
\par
Then $G$ is compactly generated if and only if $(G,d)$ is coarsely geodesic, 
if and only if $(G,d)$ is coarsely connected.
\par
Moreover, when this holds, there exists and adapted pseudo-metric $d'$ on $G$
such that $(G, d')$ is large-scale geodesic.
\end{thm}

\begin{proof}[Sketch of proof]
If $(G,d)$ is coarsely connected and $c$ is the constant given in Definition \ref{d_cg}, 
then a simple verification shows that the $c$-ball centred at $1$ in $G$ 
has a compact closure, and generates $G$.

Conversely, if $G$ is compactly generated, 
then the word metric $d'$ with respect to a given compact generating subset 
is adapted and $(G,d')$ is coarsely geodesic. 
Since it is coarsely equivalent to $(G,d)$, by Proposition \ref{TwoAdaptedOnG}, 
and since being coarsely geodesic is a coarse invariant, 
we deduce that $(G,d)$ is coarsely geodesic as well.
\end{proof}
 
Combining this with Proposition \ref{GroupsHomosCoarseIso}, 
we obtain a geometric proof of the following corollary.

\begin{cor}
\label{CgOnSub}
Let $f : G \to H$ be a continuous proper homomorphism with cocompact image 
between locally compact groups. 

Then $G$ is compactly generated if and only if $H$ is compactly generated. 
\end{cor} 
 
Also, with Corollary \ref{LargescalegeoCoarse=QI} we obtain

\begin{cor}
\label{CE=QIoncg}
Between compactly generated locally compact groups,
every coarse equivalence is a quasi-isometry.
\par

In particular, the classification of $\sigma$-compact locally compact groups up to coarse equivalence 
extends the classification of compactly generated locally compact groups up to quasi-isometry.
\end{cor}

\begin{defn}
\label{gapd}
A pseudo-metric $d$ on a compactly generated locally compact group $G$ 
is {\bf geodesically adapted} if it is equivalent to the word length $d'$ 
with respect to some/any compact generating subset, 
in the sense that the identity map $(G,d)\to (G,d')$ is a quasi-isometry.
\end{defn}
 
Analogously with Proposition \ref{TwoAdaptedOnG}, we have:

\begin{prop}
\label{anyad2}
Let $G$ be a compactly generated locally compact group. 
For any two geodesically adapted pseudo-metrics $d,d'$ on $G$, 
the identity map of pseudo-metric spaces $(G,d)\to (G,d')$ is a quasi-isometry.
\end{prop} 
 
\begin{exe}
If $G$ is a connected Lie group, 
we have two natural families of geodesically adapted pseudo-metrics:
\begin{itemize}
\item 
the metrics associated to left-invariant Riemannian metrics on $G$;
\item 
the word metrics associated to compact generating subsets of $G$ 
(observe that, by connectedness, 
any compact subset with non-empty interior generates $G$).
\end{itemize} 
Then the identity map of $G$ for any two of these metrics is a quasi-isometry, by Proposition \ref{anyad2}.
\end{exe}

\noindent
\textbf{More on proofs.}
For Theorem \ref{LargescalegeoPlusCoarselyLip} 
and Corollary \ref{LargescalegeoCoarse=QI}, see Proposition 3.B.9 in \cite{CH--16}.
For the characterizations of Theorem \ref{CaractCg}, and others, see Proposition 4.B.8 in \cite{CH--16}.
Corollaries \ref{CgOnSub} and \ref{CE=QIoncg} are then straightforward,
as well as Proposition \ref{anyad2}, which is Corollary 4.B.11 in \cite{CH--16}.

\section{Coarsely ultrametric spaces and locally elliptic locally compact groups}
\label{selli}

\begin{defn} 
A pseudo-metric space is {\bf coarsely ultrametric} if, for every $r\ge 0$, 
the equivalence relation generated by the relation ``being at distance at most $r$" 
has orbits of bounded diameter.
\end{defn}

This is a coarse invariant. Indeed, a simple verification shows the following:

\begin{prop}
\label{cultrametricspace}
A pseudo-metric space is coarsely ultrametric 
if and only if it is coarsely equivalent to an ultrametric space.
\end{prop}

Note that an immediate consequence of the definition is that,
if a pseudo-metric space is both coarsely ultrametric and coarsely geodesic, then it is bounded. 
More generally, every coarsely Lipschitz map 
from a coarsely geodesic pseudo-metric space to a coarsely ultrametric pseudo-metric space 
has a bounded image.

\begin{defn}
A locally compact group is {\bf locally elliptic} if every compact subset is contained in a compact subgroup. 
\end{defn} 
 
Note that such a locally compact group has a compact identity component. Discrete locally elliptic locally compact groups are better known as {\em locally finite groups}.
 
\begin{prop}
\label{cu=leforgroups}
If $G$ is a $\sigma$-compact locally compact group and $d$ an adapted pseudo-metric, then $G$ is locally elliptic if and only if $(G,d)$ is coarsely ultrametric. 

Among $\sigma$-compact locally compact groups, 
the class of locally elliptic groups is closed under coarse equivalence.
\end{prop} 

The verifications of the first claim is straightforward.
In the $\sigma$-compact case, the second claim folllows from the first one.

\vskip.2cm
\noindent
\textbf{More on proofs.}
For Propositions \ref{cultrametricspace} and \ref{cu=leforgroups},
see respectively 3.B.16 and 4.D.8 in \cite{CH--16}.

\section{Coarse properness, growth, and amenability}
\label{copr}

\subsection{The metric notions} 
 
\begin{defn}
The {\bf uniform growth function} of a pseudo-metric space $(X,d)$
is the function mapping $r \ge 0$ to the supremum $b_X(r)$ 
of the cardinalities of all subsets of diameter at most $r$.

A pseudo-metric space is {\bf uniformly locally finite (ULF)} 
if the function $b_X(\cdot)$ takes finite values. 
\end{defn}
 
Among non-decreasing functions $\R_+ \to \R_+$, write $f\preceq g$ 
if there exist constants $c,c',c''>0$ such that $f(r) \le cg(c'r)+c''$ for all $r>0$. 
Say that $f$ and $g$ are {\bf asymptotically equivalent}, written $f\simeq g$, if $f\preceq g\preceq f$. 

\begin{lem}
\label{grce}
If two ULF metric spaces are quasi-isometric,
they have asymptotically equivalent growth functions.
\end{lem}
 
This allows to extend the notion of growth (up to asymptotic equivalence) to a broader setting. 
 
\begin{defn}
\label{ucpspace}
A pseudo-metric space $X$ is {\bf uniformly coarsely proper}\footnote{Other authors
use ``of bounded geometry" for ``uniformly coarsely proper''.} 
if there exist a nondecreasing function $\Psi : \R_+ \to \R_+$ and $r_0>0$ such that, 
for every $r\ge r_0$, every subset of $X$ of diameter at most $r$ 
is covered by at most $\Psi(r)$ subsets of diameter at most $r_0$.  
\end{defn} 
 
Note that being uniformly coarsely proper is a coarse invariant of pseudo-metric spaces. 
More generally, if $X \to Y$ is a coarse embedding and if $Y$ is uniformly coarsely proper, 
then so is $X$; 
in the case of an isometric embedding, 
the function $\Psi$ of Definition \ref{ucpspace} can be chosen to be the same
for $X$ as for $Y$.

\begin{prop}
\label{ucp}
A pseudo-metric space is uniformly coarsely proper 
if and only if it is quasi-isometric to a ULF metric space. 
\end{prop} 

\begin{proof}[Sketch of proof]
Let us only comment the forward implication. 
Assume that $X$ is uniformly coarsely proper, with $(\Psi,r_0)$ as in the definition. 
Using Zorn's lemma, 
there exists a maximal subset $Y$ 
in which any two distinct points have distance at least $2r_0$. 
The isometric inclusion $Y \subset X$ is a quasi-isometry;
indeed any point in $X$ is at distance at most $2r_0$ of at least one point in $Y$.
Then in $Y$, for every $r \ge r_0$, any subset of diameter at most $r$ 
is covered by at most $\Psi(r)$ subsets of diameter at most $r_0$, and these are singletons. 
\end{proof}
 
\begin{defn}
If $X$ is a uniformly coarsely proper pseudo-metric space, 
the asymptotic equivalence class of the growth 
of a ULF metric space $Y$ quasi-isometric to $X$ 
is called the {\bf growth class} of $X$ 
(it does not depend on $Y$, by Lemma \ref{grce}).
\end{defn} 

Note that two quasi-isometric uniformly coarsely proper metric spaces
have the same growth class.

\vskip.2cm

In a pseudo-metric space $X$, for $Y \subset X$ and $r\ge 0$, 
denote by $B_X(Y,r)$ the set of points at distance at most $r$ to $Y$.

\begin{defn}
A ULF pseudo-metric space is called {\bf amenable} if for any $\eps>0$ and $r>0$, 
there exists a nonempty finite subset $F \subset X$ 
such that $\#(B_X(F,r))/\#(F) \le 1+\eps$.
\end{defn}

\begin{prop}
\label{UPDamenable}
Let $X,Y$ be coarsely equivalent ULF pseudo-metric spaces. 
Then $X$ is amenable if and only if $Y$ is amenable.
\end{prop}

\begin{proof}
It is enough to show that, if $Y$ is amenable, then so is $X$.
We can assume that $X$ and $Y$ are non-empty.

Let $f : X \to Y$ and $g : Y \to X$ be coarsely Lipschitz maps and $c > 0$ a constant
be such that $\sup_{x \in X} d_X(g(f(x)), x) \le c$
and $\sup_{y \in X} d_Y(f(g(y)), y) \le c$.
Let $\Phi : \R_+ \to \R_+$ be a non-decreasing function such that
$d_Y(f(x), f(x')) \le \Phi (d_X(x,x'))$ for all $x,x' \in X$.
There exist $k,\ell > 0$ such that
$\#(f^{-1}(y)) \le k$ for all $y \in Y$
and $\#(g^{-1}(x)) \le \ell$ for all $x \in X$.

Fix $r,\eps > 0$; we can suppose $r \ge c$.
Let $F \subset Y$ be a non-empty finite subset such that
$$
\#(B_Y(F, \Phi(r)+c)) / \#(F) \le 1 + \frac{\eps}{k\ell} .
$$
Define $F' = \{x \in X \mid d_Y(f(x), F) \le c \}$.
Then $F'$ contains $g(F)$, so that $\# (F') \ge \frac{1}{\ell} \# (F)$.

Let $x \in X$ be such that $0 < d_X(x, F') \le r$.
Then $c < d_Y(f(x), F) \le \Phi (r) + c$, 
that is $f(x) \in B_Y(F, \Phi(r)+c) \smallsetminus F$.
Since the cardinal of $B_Y(F, \Phi(r)+c) \smallsetminus F$
is at most $\eps \#(F) / k \ell$, 
the cardinal of $\{x \in X \mid 0 < d_X(x, F') \le r \}$
is at most $\eps \#(F) / \ell$,
and a fortiori at most $\eps \# (F')$.
It follows that the cardinal of $B_X(F', r)$ is at most
$(1 + \eps) \#(F')$.
\end{proof}

In view of Proposition \ref{UPDamenable}, the following definition is valid.
 
\begin{defn}
A uniformly coarsely proper pseudo-metric space $X$ is called 
\hfill\par\noindent 
{\bf amenable} if it is quasi-isometric to an amenable ULF metric space, 
or equivalently if every ULF metric space coarsely equivalent to $X$ is amenable.
\end{defn}

\begin{exe}
\label{exsubexp}
If $X$ is a nonempty ULF pseudo-metric space and 
\hfill\par\noindent
$\liminf_{n \to \infty} b_X(n+1)/b_X(n)=1$, then $X$ is amenable. 

If $X$ is a non-empty coarsely uniformly proper metric pseudo-space 
of subexponential growth,
 then $X$ is amenable.
 \end{exe}
 
 \noindent
\textbf{On proofs.}
For
Lemma \ref{grce}, 
Proposition \ref{ucp}, 
Proposition \ref{UPDamenable}, 
and Example \ref{exsubexp}, 
see \cite{CH--16}, respectively
Propositions 3.D.6, 3.D.16, 3.D.33, and Example 3.D.38.
 
\subsection{The case of locally compact groups}
 
\begin{thm}
\label{Gucp}
Let $G$ be a $\sigma$-compact locally compact group and $d$ an adapted pseudo-metric (Definition \ref{apd}). Then $(G,d)$ is uniformly coarsely proper. 
\end{thm}

In particular, the notion of metric amenability 
makes sense for any $\sigma$-compact locally compact group. 
The notion of growth (up to asymptotic equivalence) 
makes sense for any compactly generated locally compact group, 
by considering the growth of $(G,d)$ for a geodesically adapted pseudo-metric $d$ on $G$,
in the sense of Definition \ref{gapd}.

It can also be shown that this notion of growth is equivalent to that
involving the Haar measures of balls in the group.

\begin{defn}
A $\sigma$-compact locally compact group $G$ is {\bf geometrically amenable} 
if $(G,d)$ is amenable.
\end{defn}
 
This is closely related, but not equivalent, to the notion of amenability. 
Recall that a locally compact group $G$, 
endowed with a {\em left} Haar measure $\lambda$, is {\bf amenable} if, 
for every compact subset $S$ and every $\eps>0$, 
there exists a measurable subset $F$ of finite nonzero measure
such that $\lambda(SF) \le (1+\eps)\lambda(F)$. 

Besides, recall that a locally compact group $G$ is {\bf unimodular} 
if the action of $G$ on itself by conjugation preserves some 
(and hence every) left Haar measure.

\begin{prop}
\label{Gmetam}
A $\sigma$-compact, locally compact group is geometrically amenable 
if and only if it is amenable and unimodular.
\end{prop}

\begin{proof}[On the proof] 
Let us say that a locally compact group is {\bf right-amenable} 
if, for every compact subset $S$ and every $\eps>0$, 
there exists a measurable subset $F$ of finite nonzero measure 
such that $\lambda(FS) \le (1+\eps)\lambda(F)$. 
Note that, in comparison with amenability, $SF$ has been replaced by $FS$, 
while we still have a {\em left} Haar measure. 
The subset $FS$ can be thought of as a metric thickening of $F$, 
and a routine verification shows that a $\sigma$-compact locally compact group 
is geometrically amenable if and only if it is right-amenable. 
Now on the one hand, for a {\em unimodular} group, 
it is clear that amenability and right-amenability are equivalent properties. 
On the other hand, if a locally compact $G$ is not unimodular, 
if $s$ is an element with $\Delta(s)>1$, so that $\lambda(Fs)=\Delta(s)\lambda(F)$, 
the condition of right-amenability fails for $S=\{s\}$.
\end{proof}
 
\begin{cor}
\label{A+Uisci}
To be amenable and unimodular is a coarse invariant among locally compact groups. 
\par
In particular, to be amenable is a coarse invariant among discrete groups. 
\end{cor}

Note that this is not true when unimodularity is dropped. 
Indeed, there are many cocompact closed inclusions of groups $H \subset G$ 
with $H$ amenable (necessarily non-unimodular), and $G$ non-amenable.
Let us indicate two examples,
with $n\ge 2$ and $\K$ a non-discrete locally compact field, e.g.\ $\K = \R$:
\begin{itemize}
\item[(1)] 
$G = \GL_n(\K)$, $H = \mathrm{T}_n(\K)$, the subgroup of upper triangular matrices;
\item[(2)]
$G = \GL_n(\K) \ltimes \K^n$ (the group of affine transformations), 
$H = \mathrm{T}_n(\K) \ltimes \K^n$. 
\end{itemize}
Observe that $G$ is unimodular in the first example, and non-unimodular in the second.

\vskip.2cm

\noindent
\textbf{On proofs.} 
For Theorem \ref{Gucp} and Proposition \ref{Gmetam}, 
see Propositions 3.D.29 and 4.F.5 in \cite{CH--16}.
Corollary \ref{A+Uisci} is a straightforward consequence of Proposition \ref{Gmetam}.

\section{Compactly presented groups}
\label{coprg}

\subsection{Coarsely simply connected metric spaces} 
 
Let $X$ be a pseudo-metric space, $c$ a positive real number,
and $k$ is a positive integer.
The {\bf Rips complex} $\Rips_c^k(X)$ is the simplicial complex whose set of vertices is $X$, 
and a subset $Y \subset X$ forms a simplex if its cardinal is at most $k+1$ 
and its points are pairwise at distance $\le c$. 
The $k$-simplices are endowed with the metric 
induced by the standard $\ell^\infty$-norm on $\R^{k+1}$.
\par

For instance, the pseudo-metric space $X$ is coarsely connected (Definition \ref{d_cg}) 
if and only if $\Rips_c^1(X)$ is connected for some $c$ 
(then $\Rips_{c'}^k(X)$ is connected for all $c'\ge c$ and $k\ge 1$).

\begin{defn}
The pseudo-metric space $X$ is {\bf coarsely simply connected} 
if there exist $c \ge 0$ and $c' \ge c$  such that $\Rips_c^1(X)$ is connected 
and every loop in $\Rips_c^1(X)$ is homotopically trivial in $\Rips_{c'}^2(X)$.
\end{defn}

It is possible to interpret the latter condition by a certain discrete connectedness property, 
along with the requirement that every discrete path in $X$ 
has a discrete homotopy to the trivial loop. 
The precise statement is technical and we refer to \cite{CH--16}.

\begin{prop}
\label{csccoarseinv}
To be coarsely simply connected is a coarse invariant of pseudo-metric spaces.
\end{prop}

\subsection{Compactly presented groups}

\begin{defn}
A {\bf bounded presentation} is a presentation of the form $\langle S\mid R\rangle$, 
where $S$ is an arbitrary set and $R \subset F_S$ 
(where $F_S$ is the free group over $S$) is a set of words of bounded length with respect to $S$. 

A locally compact group $G$ is {\bf compactly presented} 
if there exists an isomorphism of a boundedly presented group $\langle S\mid R\rangle$ onto $G$ 
such that the image of $S$ is a compact generating subset of $G$.
\end{defn}

In other words, $G$ has a presentation by a compact subset of generators 
and relators of bounded length. 

\begin{exe}
A discrete group is compactly presented if and only if it is finitely presented. 
\end{exe} 

Recall that, if $S$ is a generating subset of a group, 
the Cayley graph $\mathcal{G}(G,S)$ is the graph 
whose set of vertices is $G$ and for which $(g,h)$ is an edge 
whenever $g^{-1}h \in S \cup S^{-1}$.
Observe that $\mathcal{G}(G,S)$ is $\Rips_1^1(X)$
when $(X,d) = (G, d_S)$.

Standard homotopy arguments show the following:

\begin{prop}
\label{cpviaCayley}
Let $G$ be a compactly generated locally compact group; 
consider its Cayley graph $\mathcal{G}(G,S)$ 
with respect to some compact generating subset $S$. 

Then $G$ is compactly presented if and only if $\pi_1(\mathcal{G}(G,S))$ 
is generated by loops of bounded size, 
in other words if and only if $\mathcal{G}(G,S)$ can be filled in a $G$-invariant way 
by gons of bounded size so that the resulting 2-complex is simply connected.
\end{prop}

Here, ``loops of bounded size" more precisely means 
loops of the form $\gamma c\gamma^{-1}$, for paths $\gamma$ starting from $1$, 
and loops $c$ of bounded diameter, based at the end of $\gamma$.

\begin{prop}
\label{cpco}
Let $G$ be a $\sigma$-compact locally compact group 
and $d$ an adapted pseudo-metric on $G$. 
Then $G$ is compactly presented if and only if $(G,d)$ is coarsely simply connected. 
\par
In particular, to be compactly presented is invariant under coarse equivalence 
among $\sigma$-compact locally compact groups.
\end{prop}

Standard facts about finitely presented groups carry over to compactly presented groups.

\begin{prop}
\label{modn}
Let $G$ be a locally compact group, $N$ a closed normal subgroup and $Q=G/N$. 
\begin{enumerate}
\item If $G$ is compactly presented and $N$ is compactly generated qua normal subgroup, 
then $Q$ is compactly presented;
\item if $N$ and $Q$ are compactly presented then so is $G$;
\item if $G$ is compactly generated and $Q$ is compactly presented, 
then $N$ is compactly generated qua normal subgroup.
\end{enumerate}
\end{prop}
 
There are no free groups in the context of locally compact groups. 
Nevertheless, we have the following:

\begin{prop}
\label{cgquotientcp}
Every compactly generated locally compact group $Q$ is isomorphic 
to the quotient of some compactly presented locally compact group $G$ 
by a discrete normal subgroup $N$.
\end{prop}

Note that, by Proposition \ref{modn}, 
$Q$ is compactly presented if and only if $N$ is finitely generated as a normal subgroup.

In the case of totally disconnected groups, this can be refined.

\begin{defn}
A {\bf tree-like} locally compact group is a locally compact group 
admitting a proper cocompact action on some tree of bounded valency.
\end{defn} 
 
If the group is assumed to be compactly generated, it can be shown that ``cocompact" 
can be removed from the definition. 
Note that a finitely generated group is tree-like if and only if it is virtually free.

\begin{thm}
\label{quotienttreelike}
Every compactly generated, totally disconnected 
(or more generally, which a compact unit component) locally compact group 
is isomorphic to the quotient of some tree-like locally group by some discrete normal subgroup.
\end{thm}

\noindent
\textbf{On proofs.} For Propositions
\ref{csccoarseinv}, 
\ref{cpco}, 
\ref{modn}, 
\ref{cgquotientcp}, 
and Theorem \ref{quotienttreelike}, 
see \cite{CH--16}, respectively 
Propositions 6.A.7, 8.A.3,  8.A.10, Corollary 8.A.17, 
and Theorem 8.A.20.
Cayley graphs hardly appear in \cite{CH--16},
but proving Proposition \ref{cpviaCayley} 
is an easy exercise.

\section{The Bieri-Strebel Theorem}
\label{bst} 

Let $H$ be a locally compact group, $K, L$ two open subgroups,
and $\varphi : K \overset{\simeq}{\longrightarrow} L$ an isomorphism
of topological groups. 
On the resulting HNN-extension
\begin{equation*}  
\HNN(H, K, L, \varphi) \, = \, 
\langle H, t \mid t k t^{-1} = \varphi(k) \hskip.2cm \forall \hskip.1cm k \in K \rangle ,
\end{equation*}
there exists a unique topology making it
a topological group in which $H$ is an open subgroup;
moreover, this topology is locally compact
(if necessary, see Proposition 8.B.10 in \cite{CH--16}).
A locally compact group $G$ 
\textbf{splits as an HNN-extension} over an open subgroup $H$
if there exist $K, L, \varphi$ as above such that, as a pair of topological groups, 
$(G, H)$ is isomorphic to $(\HNN(H, K, L, \varphi), H)$.

\begin{thm}
\label{BSthm}
Let $G = G_\infty$ be a compactly generated locally compact group 
with a continuous homomorphism $\pi=\pi_\infty$ of $G$ onto $\Z$. 
Then there exist
\begin{itemize}
\item 
a sequence $(G_n)_{n\ge 0}$ of locally compact groups, 
with surjective continuous homomorphisms $\pi_n:G_n\to\Z$,
\item 
surjective continuous homomorphisms $\varphi_{m,n} : G_m \to G_n$ 
with discrete kernels, for $m \le n \le \infty$,
\end{itemize}
such that
\begin{itemize}
\item
the $\varphi_{m,n}$ are compatible with each other 
($\varphi_{m,n} \circ \varphi_{\ell,m} = \varphi_{\ell,n}$ for all $\ell \le m \le n \le \infty$) 
and compatible with the projections 
($\pi_m = \pi_n \circ \varphi_{m.n}$ for all $m \le n \le \infty$), 
\item 
$G_n$ splits as an HNN-extension 
over some compactly generated open subgroup of $\mathrm{Ker}(\pi_n)$,
for all $n < \infty$.
\end{itemize}
\end{thm} 
 
This theorem is an approximation theorem. 
Note that, when $\mathrm{Ker}(\pi)$ is compactly generated, 
it is an empty statement, since we can choose $G_n=G$ for all $n$. 
However, it provides useful information 
when $\mathrm{Ker}(\pi)$ is not assumed to be compactly generated.
When $G$ is compactly presented,
then $\varphi_{\infty,n}$ has to be an isomorphism for some $n < \infty$, which provides
a version for locally compact groups of a theorem of Bieri and Strebel
(see \cite[Theorem A]{BiSt--78} and \cite[Proposition I.3.2]{Abel--87}):

\begin{cor}[Bieri-Srebel splitting theorem]
\label{BSsplits}
Let $G$ be a compactly presented locally compact group 
along with a continuous homomorphism $\pi$ of $G$ onto $\Z$. 
\par
Then $G$ splits as an HNN-extension 
over some compactly generated open subgroup of $\mathrm{Ker}(\pi)$.
\end{cor}

Let $G = \HNN (H,K,L, \varphi)$ be an HNN-extension as above.
In the particular case of $K=H$, 
the subgroup $N := \bigcup_{n \ge 0} t^{-n}Ht^n$ is open in $G$,
the endomorphism $\varphi$ of $H$ extends to an automorphism $\alpha$ of $N$
by $\alpha (x) = txt^{-1}$ for all $x \in N$,
and $G$ is naturally isomorphic to the semi-direct product
$N \rtimes_\alpha \Z$, where $n \in \Z$ acts on $N$ by $\alpha^n$.

When, on the contrary, the HNN-extension is
\textbf{non-ascending}, i.e.\ when $K \ne H \ne L$,
then $G$ contains a non-abelian discrete free subgroup;
hence it follows from Corollary \ref{BSsplits} that:

\begin{cor}
\label{hnna} 
Let $G = N \rtimes_\alpha \Z$ be a compactly presented locally compact group, 
with $\Z$ acting through powers of some topological group automorphism $\alpha$ of $N$.  
Assume that $G$ has no non-abelian discrete free subgroup. 
\par

Then one of $\alpha, \alpha^{-1}$ engulfs $N$ 
into some compactly generated open subgroup of $N$.
\end{cor}

We have used:

\begin{defn}
An automorphism $\alpha$ of a group $N$ \textbf{engulfs} $N$ into a subgroup $H$ of $N$
if $\alpha(H) \subset H$ and $\bigcup_{n \ge 1} \alpha^{-n}(H) = N$.
\end{defn}

\noindent
\textbf{On proofs.}  
For Theorem \ref{BSthm}, 
Corollary \ref{BSsplits}, 
Corollary \ref{hnna}, 
see \cite{CH--16}, respectively
Theorems 8.C.8, 8.C.3, and Proposition 8.C.18.

\section{Examples}
\label{sexam}

In this section, we provide various examples of locally compact groups
that are compactly presented, and some that are not.

\begin{prop}
\label{ggci}
If $G$ is a locally compact group and $G/G^\circ$ is compact, then $G$ is compactly presented. 
\end{prop}

In the situation of Proposition \ref{ggci}, 
$G$ can be shown to admit a proper transitive continuous action 
by isometries on some Riemannian manifold homeomorphic to some Euclidean space.

\begin{prop}
\label{exnilp}
Every nilpotent compactly generated locally compact group is compactly presented. 
\par
More generally, a locally compact group that is compactly generated
and of polynomial growth is compactly presented.
\end{prop} 

Indeed, given a locally compact group $G$ that is compacly generated
and of polynomial growth, it is shown in \cite[Theorem 1.2]{Breu--14} that there exists 
a closed and cocompact subgroup $H$ of $G$
and a proper homomorphism with cocompact image of $H$
into a connected Lie group $L$ (moreover L is simply connected, solvable, 
and of polynomial growth).
Since $L$ is compactly presented by Proposition \ref{ggci},
$G$ is compactly presented by Propositions \ref{cpco} and \ref{GroupsHomosCoarseIso}.

\vskip.2cm

Consider an ultrametric non-discrete locally compact field $\K$
(for example the field $\Q_p$ of $p$-adic numbers, for some prime $p$),
the Heisenberg group $H(\K)$ of triples of elements of $\K$,
with product defined by
$(x,y,t)(x',y',t') = (x+x', y+y', t+t'+xy'-x'y)$,
the action of $\SL_2(\K)$ on $H(\K)$ defined by
\begin{equation*}
\begin{pmatrix}
a & b \\
c & d 
\end{pmatrix}
(x,y,t) = (ax+by, cx+dy, t) ,
\end{equation*}
and the corresponding semi-direct product $G := H(\K) \rtimes \SL_2(\K)$.
The centre $Z$ of $G$ is isomorphic to that of $H(\K)$, 
i.e.\ to the additive group of $\K$,
and the quotient $G/Z$ is isomorphic to the natural semi-direct product
$\K^2 \rtimes \SL_2(\K)$.
It is easy to check that $G$ is compactly generated;
since $Z$ is not compactly generated,
it follows from Proposition \ref{modn}(3) that:

\begin{prop}
\label{H(K)sSL2(K)}
For every ultrametric non-discrete locally compact field $\K$,
the semi-direct product $\K^2 \rtimes \SL_2(\K)$
is compactly generated and is not compactly presented.
\end{prop}

\begin{prop}
\label{phnna}
Let $H$ be a compactly presented locally compact group 
and $\varphi$ an injective continuous endomorphism of $H$ with open image. 
Then the ascending HNN-extension associated to $(H,\varphi)$ is compactly presented.
\end{prop}

In the situation of Proposition \ref{phnna}, 
denote by $G$ the HNN-extension, and let $N, \alpha$ be defined 
as before Corollary \ref{hnna}.
Then $G \simeq N \rtimes_\alpha \Z$,
and this shows that
Proposition \ref{phnna} is a particular case of the following one:

\begin{prop}
\label{tamet}
Consider a locally compact group 
with a topological semidirect product decomposition $G=N\rtimes\Z^k$, 
such that some element $\alpha$ of $\Z^k$ engulfs $N$ 
into some compactly presented open subgroup of $N$.
\par
Then $G$ is compactly presented.
\end{prop}

\noindent
\textbf{On proofs.} For 
Propositions \ref{ggci}, 
\ref{exnilp}, 
\ref{H(K)sSL2(K)}, 
\ref{phnna}, 
\ref{tamet}, 
see \cite{CH--16}, respectively
Propositions 8.A.13, 8.A.22, 
8.A.28, 8.B.10, and Lemma 8.D.7.

\begin{prop}
\label{solt}
Let $N_1,N_2$ be totally disconnected non-compact locally compact groups.
For $i = 1,2$, assume that there exist
a topological group automorphism $\alpha_i$ of $N_i$
engulfing $N_i$ into some compact open subgroup $H_i$ of $N_i$.
Consider the automorphism of $N_1 \times N_2$ given by 
$\alpha = (\alpha_1,\alpha_2^{-1})$. 
\par
Then the semidirect product $(N_1 \times N_2)\rtimes_\alpha \Z$ 
is not compactly presented.
\end{prop}

\begin{lem}
\label{fichtre}
Let $N$ be a non-compact locally compact group
and $\beta$ a topological group automorphism engulfing $N$
into some compact open subgroup $H$ of $N$.
Set $K = \bigcap_{n \ge 0} \beta^n (H)$.
\par
For all $x \in N$ with $x \notin K$ and for all compact subset $C$ of $N$,
we have $\beta^{-n}(x) \notin C$ for $n$ large enough.
\end{lem}

\begin{proof}
There exists $n_1 \ge 0$ such that $C \subset \beta^{-n_1}(H)$,
because $\left( \beta^{-n}(H) \right)_{n \ge 0}$ is an open covering of $N$.
There exists $n_2 \ge 0$ such that $x \notin \beta^{n_2}(H)$, 
because $x \notin K$.
For any $n \ge n_1 + n_2$, we have therefore
$\beta^{-n}(x) \notin \beta^{-(n-n_2)}(H)$;
since $C \subset \beta^{-n_1}(H) \subset \beta^{-(n-n_2)}(H)$,
we have also $\beta^{-n}(x) \notin C$.
\end{proof}

\begin{proof}[On the proof of Proposition \ref{solt}]
For $i = 1,2$, the group $N_i$ is locally elliptic.  
Indeed, for every compact subset $C$ of $N_i$,
there exists an integer $n \ge 1$ such that $\alpha_i^n(C) \subset H_i$.
Hence $N:= N_1 \times N_2$ is locally elliptic.
\par

Suppose by contradiction that $G:= N \rtimes_\alpha \Z$ is compactly presented.
Then $G$ splits as an HNN-extension over some open subgroup of $N$
that is compactly generated by Corollary \ref{BSsplits}, 
and therefore compact by local ellipticity.
By Corollary 8.C.19 of \cite{CH--16}, 
it follows that $\alpha^{\varepsilon}$ engulfs $N$ into some
compact open subgroup $H$ of $N$, for an appropriate $\varepsilon \in \{1, -1\}$.
In particular, for every $x \in N$, we have $\alpha^{\varepsilon n}(x) \in H$ for $n$ large enough.
\par

For $i=1,2$, set $K_i = \bigcap_{n \ge 0} \alpha_i^n (H_i)$. 
Choose $x = (x_1, x_2) \in N = N_1 \times N_2$ with $x_1 \notin K_1$ and $x_2 \notin K_2$.
By Lemma \ref{fichtre}, for every compact subset $C$ of $N$ (for example for $C=H$),
we have $\alpha^{-\varepsilon n}(x) \notin C$ and $\alpha^{\varepsilon n}(x) \notin C$
for $n$ large enough.
As this contradicts the conclusion of the last paragraph,
$G$ cannot be compactly presented.
\end{proof}

\begin{exe}
\label{k1k2}
Let $\K_1,\K_2$ be ultrametric non-discrete locally compact fields,
given together with their canonical absolute value 
$\K_i \ni \lambda \mapsto \vert \lambda \vert \in \R_+$ (for $i=1,2$). 
Fix $\lambda_1 \in \K_1^*$ and $\lambda_2 \in \K_2^*$. 
Consider the semidirect product $G=(\K_1 \times \K_2)_{(\lambda_1,\lambda_2)}\rtimes\Z$,
with respect to the action of $\Z$ defined by $a$

Then
\begin{itemize}
\item 
If either $\vert \lambda_1 \vert$ or $\vert \lambda_2 \vert$ is equal to 1, 
then $G$ is not compactly generated.
\item 
if $|\lambda_1|<1<|\lambda_2|$ or $|\lambda_2|<1<|\lambda_1|$, then $G$ is compactly generated but not compactly presented (as a particular case of Proposition \ref{solt})
\item 
if $|\lambda_1|$ and $|\lambda_2|$ are both $<1$ or both $>1$, 
then $G$ is compactly presented (by Corollary \ref{hnna}).
\end{itemize}
In particular, given two primes $p, q$, consider the action $\alpha(p^k, q^{-\ell})$ of $\Z$
on $\Q_p \times \Q_q$ for which $1$ acts by $(x,y) \mapsto (p^k x, q^{-\ell} y)$,
for some positive integers $k, \ell$;
then  the group $(\Q_p \times \Q_q) \rtimes_{\alpha(p^k, q^{-\ell})} \Z$ is not compactly presented.
\end{exe}

\begin{exe}
\label{exBSmotiv}
Consider two distinct primes $p,q$ and, for $i=1,2$, 
the semidirect product $\Gamma_i = \Z[1/pq] \rtimes_{n_i} \Z$, 
where $\Z$ acts by multiplication by $n_i$, with $n_1=pq$ and $n_2=p/q$. 
Then both $\Gamma_1$ and $\Gamma_2$ are finitely generated. 
\par

Consider moreover the locally compact group 
$G_i = (\R \times \Q_p \times \Q_q) \rtimes_{n_i} \Z$,
where $\rtimes_{n_i}$ indicates that $1 \in \Z$ acts by multiplication by $n_i$
on each of the three factors $\R, \Q_p, \Q_q$.
The group $\Gamma_i$ is naturally a cocompact lattice in $G_i$.

\par

It follows from Example \ref{k1k2} that $G_1/\R$ is compactly presented 
and $G_2/\R$ is not compactly presented. 
Since $\R$ itself is compactly presented, it follows from Theorem \ref{modn} 
that $G_1$ is compactly presented and $G_2$ is not compactly presented. 
By Proposition \ref{GroupsHomosCoarseIso}, the inclusion of $\Gamma_i$ into $G_i$ is a coarse equivalence. 
By Proposition \ref{cpco}, we deduce that $\Gamma_1$ is finitely presented while $\Gamma_2$ is not.
\end{exe}

\begin{thm}[Behr]
\label{exredcp}
If $\mathbf{G}$ is a reductive $\K$-group, 
for some non-discrete locally compact field $\K$, 
the group $G = \mathbf{G}(\K)$ of $\K$-points of $\mathbf{G}$ is compactly presented. 
\end{thm} 

Here is the strategy for a proof, different from that in \cite{Behr--67}. 
When $\K$ is Archimedean, then $G$ has finitely many connected components 
and Proposition \ref{ggci} applies.
Otherwise, $G$ admits some closed cocompact (solvable) subgroup 
satisfying the hypotheses of Proposition \ref{tamet};
details for the simpler case of $\SL_n(\K)$ 
can be found in the proof of Theorem 8.D.12 in \cite{CH--16}.

\begin{cor}
Let $G$ be as in Theorem \ref{exredcp}.
Every cocompact lattice in $G$ is finitely presented.
\end{cor}

In case $\K$ is non-Archimedean, recall that every lattice in $G$
is cocompact \cite{Tama--65}.

\vskip.4cm

\noindent
Y.C.: Laboratoire de  Math\'ematiques d'Orsay \\
Universit\'e Paris-Sud,
CNRS \\
Universit\'e de Paris-Saclay \\
91405 Orsay, France \\
yves.cornulier@math.u-psud.fr

\vskip.4cm

\noindent
P. H.: Section de math\'ematiques \\
Universit\'e de Gen\`eve, C.P.~64 \\
CH--1211 Gen\`eve~4, Suisse \\
Pierre.delaHarpe@unige.ch

\end{document}